\documentclass[12pt,a4paper,reqno]{amsart}

\usepackage{amsfonts}
\usepackage{amsmath,amssymb,amsthm,amsxtra}
\usepackage[utf8]{inputenc}
\usepackage[T1]{fontenc}
\usepackage[dvipsnames]{xcolor}
\usepackage{graphicx}
\usepackage{tikz-cd}
\usepackage{tabularx}
\usepackage{float}
\usepackage[colorlinks, linkcolor = MidnightBlue, anchorcolor = Periwinkle,
citecolor = purple, urlcolor = Emerald]{hyperref}
\usepackage{enumitem}
\usepackage{fullpage}
\usepackage{url}

\usepackage{setspace}
\usepackage{extarrows}
\usepackage{cleveref}
\usepackage{mathtools}

\theoremstyle{plain}
\newtheorem{theorem}{Theorem}[section]

\newtheorem{proposition}[theorem]{Proposition}
\newtheorem{lemma}[theorem]{Lemma}
\newtheorem{corollary}[theorem]{Corollary}

\theoremstyle{definition}
\newtheorem{definition}[theorem]{Definition}

\newtheorem{remark}[theorem]{Remark}

\numberwithin{equation}{section}

\DeclareMathOperator{\Hilb}{Hilb}

\DeclareMathOperator{\End}{End}
\DeclareMathOperator{\ParEnd}{ParEnd}
\DeclareMathOperator{\SParEnd}{SParEnd}
\DeclareMathOperator{\Ext}{Ext}
\DeclareMathOperator{\Hom}{Hom}
\DeclareMathOperator{\Spec}{Spec}

\DeclareMathOperator{\Aut}{Aut}
\DeclareMathOperator{\tr}{tr}
\DeclareMathOperator{\Pic}{Pic}
\DeclareMathOperator{\diag}{diag}

\DeclareMathOperator{\pardeg}{pardeg}
\DeclareMathOperator{\rk}{rk}

\newcommand{\mbC}{\mathbb{C}}

\newcommand{\mbA}{\mathbb{A}}
\newcommand{\mbP}{\mathbb{P}}
\newcommand{\mcM}{\mathcal{M}}

\newcommand{\mcO}{\mathcal{O}}

\newcommand{\mcK}{\mathcal{K}}

\newcommand{\mcV}{\mathcal{V}}

\newcommand{\wt}{\widetilde}

\title{On the Symplectic forms of Groechenig's Higgs moduli over an elliptic curve}

\date{\today}

\author{Zelin Jia}
\address{Graduate School of Mathematics\\ Nagoya University\\Furocho, Chikusa-ku\\ Nagoya\\464-8602\\ Japan.}
\email{zelin.jia.c0@math.nagoya-u.ac.jp}

\begin{document}

\begin{abstract}
Gorsky, Nekrasov, and Rubtsov introduced the moduli space of marked Higgs bundles over an elliptic curve $E$
and identified it with the Hilbert scheme of points on the cotangent bundle.
Later, Groechenig constructed four analogous isomorphisms on marked rational curves, 
of affine Dynkin types $\wt D_4$, $\wt E_6$, $\wt E_7$, and $\wt E_8$, 
for the $\Gamma$-Hilbert schemes of $T^*E$ for cyclic groups $\Gamma$ with $|\Gamma|\in\{2,3,4,6\}$.
We prove that the isomorphisms in all five cases are holomorphic symplectomorphisms. 
\end{abstract}

\maketitle
\tableofcontents

\section{Introduction}

Let $E$ be a complex elliptic curve with chosen origin $o$, 
and let $\Gamma\subset \Aut(E,o)$ be a finite cyclic group such that
\[
        |\Gamma|\in\{1,2,3,4,6\}.
\]
Let us denote
\[
        C_\Gamma=
        \begin{cases}
        E,&|\Gamma|=1,\\
        E/\Gamma\simeq\mbP^1,&|\Gamma|>1,
        \end{cases}
        \qquad
        p_\Gamma:E\longrightarrow C_\Gamma,
\]
where $p_\Gamma$ is the identity map for $|\Gamma|=1$ and the quotient map for the other four cases,
and we write
\[
        0=p_\Gamma(o)\in C_\Gamma.
\]
For $|\Gamma|=2$ the group is generated by involutions on an arbitrary elliptic curve.  
For $|\Gamma|=4$ one uses the elliptic curve with Gaussian complex multiplication, 
and for $|\Gamma|\in\{3,6\}$ one uses the elliptic curve with Eisenstein complex multiplication.  
In the coordinate centered at $o$,
we can write the action as $\gamma \cdot z=\zeta_\gamma z$, 
and we define the induced action on the cotangent bundle $T^*E$ by
\[
 \gamma\cdot\bigl(z,\xi\,dz|_z\bigr)
 =\bigl(\zeta_\gamma z,
        \zeta_\gamma^{-1}\xi\,dz|_{\zeta_\gamma z}\bigr);
\]
For simplicity, 
we write $\gamma\cdot(z,\xi)=(\zeta_\gamma z,\zeta_\gamma^{-1}\xi)$.
Thus the quotient
\[
        X_\Gamma:=T^*E/\Gamma
\]
is $T^*E$ itself for $|\Gamma|=1$ and is a normal surface with rational double points for $|\Gamma|>1$.  
Therefore, the $\Gamma$-Hilbert scheme
\[
        S_\Gamma:=\Gamma\text{-}\Hilb(T^*E)
\]
equals $T^*E$ for $|\Gamma|=1$ and is the crepant resolution of $X_\Gamma$ for $|\Gamma|>1$.

For $|\Gamma|=1$ and every $m\geq1$,
Gorsky, Nekrasov and Rubtsov \cite{GNR} gives an isomorphism from $(T^*E)^{[m]}$ to the $\wt A_0$-type moduli space of marked, 
equivalently parabolic, Higgs bundles on $E$.
For $|\Gamma|>1$,
Groechenig \cite{Groechenig} proved that $S_\Gamma$ is isomorphic to a two-dimensional moduli space of stable parabolic Higgs bundles on a marked rational curve,
where the isomorphism is given by the relative Fourier--Mukai transform.
The four possible non-trivial cyclic groups correspond to the affine Dynkin diagrams
\[
        \wt D_4,\qquad \wt E_6,\qquad \wt E_7,\qquad \wt E_8.
\]
He further proved that for every $m\geq 1$ the Hilbert scheme $S_\Gamma^{[m]}$ is again a moduli space of stable parabolic Higgs bundles.  

The purpose of this paper is to prove the symplectomorphism statement for the above five cases.

\begin{theorem}[Main theorem]\label{thm:main}
  Let $\Gamma\subset \Aut(E,o)$ be a finite cyclic group with
  $|\Gamma|\in\{1,2,3,4,6\}$ and let
    \[
        \Phi_{\Gamma,m}:S_\Gamma^{[m]}\xrightarrow{\sim}\mcM_{\Gamma,m}
    \]
  be the isomorphism from the Hilbert scheme of $m$ points on $S_\Gamma=\Gamma\text{-}\Hilb(T^*E)$ to the corresponding stable parabolic Higgs moduli space on $C_\Gamma$
  constructed by Groechenig \cite{Groechenig}.  
  Then we have
    \[
        \Phi_{\Gamma,m}^*\Omega_{\mcM_{\Gamma,m}}=\omega_{S_\Gamma^{[m]}},
    \]
  where $\Omega_{\mcM_{\Gamma,m}}$ is the symplectic form on the moduli space of stable parabolic Higgs bundles constructed by deformation theory, 
  and $\omega_{S_\Gamma^{[m]}}$ is the holomorphic symplectic form induced by the symplectic form on $S_\Gamma$.
\end{theorem}

We prove the theorem by an explicit calculation of the symplectic forms on the reduced spectral locus 
and then extend the equivalence globally.

As a consequence, as in Section~\ref{sec:hitchin-maps},
we can construct an algebraic integrable system as Hitchin system on the Hilbert scheme of points:
\[
 h_{\Gamma,m}:S_\Gamma^{[m]}\longrightarrow\mbA^m.
\]

\subsection{The parabolic data}

We use the following explicit parabolic descriptions.
In every row, the Higgs field has at most order one poles along the marked divisor, 
and its residues are nilpotent with respect to the listed flags.
In the $\wt A_0$ case, the extra one-dimensional subspace is inserted as ``$1$'' in the parabolic flag over $o\in E$.  
In each nontrivial quotient case,
it is inserted as ``$1$'' in the parabolic flag over the marked point $0=p_\Gamma(o)\in C_\Gamma$.

\begin{table}[H]
\centering
\small
\renewcommand{\arraystretch}{1.24}
\setlength{\tabcolsep}{4pt}
\begin{tabularx}{\textwidth}{c|c|c|c|X}
\textup{type} & $|\Gamma|$ & \textup{coarse curve} & \textup{rank} & \textup{flag dimensions at marked points} \\
\hline
$\wt A_0$ & $1$ & $E$ & $m$ & $(m,1,0)$ at the origin $o$ of $E$ \\
$\wt D_4$ & $2$ & $\mbP^1$ & $2m$ & \begin{tabular}[t]{@{}l@{}}$(2m,m,0)$, $(2m,m,0)$,\\ $(2m,m,0)$, $(2m,m,1,0)$\end{tabular} \\
$\wt E_6$ & $3$ & $\mbP^1$ & $3m$ & \begin{tabular}[t]{@{}l@{}}$(3m,2m,m,0)$, $(3m,2m,m,0)$,\\ $(3m,2m,m,1,0)$\end{tabular} \\
$\wt E_7$ & $4$ & $\mbP^1$ & $4m$ & \begin{tabular}[t]{@{}l@{}}$(4m,2m,0)$, $(4m,3m,2m,m,0)$,\\ $(4m,3m,2m,m,1,0)$\end{tabular} \\
$\wt E_8$ & $6$ & $\mbP^1$ & $6m$ & \begin{tabular}[t]{@{}l@{}}$(6m,3m,0)$, $(6m,4m,2m,0)$,\\ $(6m,5m,4m,3m,2m,m,1,0)$\end{tabular}
\end{tabularx}
\caption[The five Hilbert-scheme/parabolic-Higgs families]{The five Hilbert-scheme/parabolic-Higgs families.}
\end{table}
The divisors, the corresponding flags and the stability conditions are described explicitly in Section~\ref{sec:parabolic_form}.

\section{The parabolic Higgs bundles}

In this section,  following \cite{Groechenig}, we first describe the
quotient surface $T^*E/\Gamma$ and its crepant resolution $S_\Gamma$.  We then
specify the corresponding stable strongly parabolic Higgs bundles, including their
flags, weights, and stability conditions.

Finally, we will recall the deformation complex 
and the resulting holomorphic symplectic form on their moduli spaces.

\subsection{Quotient surfaces and the unique crepant resolution}

We let $z$ be a local coordinate on $E$
and let $\xi$ be the dual cotangent coordinate, 
so that a cotangent vector is written as $\xi\,dz$.
Then we can write the tautological one-form and symplectic form on $T^*E$ as
\[
        \lambda_E=\xi\,dz,
        \qquad
        \omega_E=d\lambda_E=d\xi\wedge dz.
\]
If $\gamma\in\Gamma$ acts on $E$ by $z\mapsto \zeta z$ in a local coordinate around a fixed point, 
then the induced action on the cotangent bundle acts by
\[
        (z,\xi)\longmapsto (\zeta z,\zeta^{-1}\xi).
\]
Which implies that
\[
        (\zeta^{-1}\xi)\,d(\zeta z)=\xi\,dz,
\]
thus we have $\gamma^*\lambda_E=\lambda_E$ and $\gamma^*\omega_E=\omega_E$.

We let
\[
        q_\Gamma:T^*E\longrightarrow X_\Gamma:=T^*E/\Gamma
\]
be the quotient map.
Denote by $(T^*E)^{\circ}$ the open locus with trivial stabilizer, 
i.e. the surface minus the torsion points.
Denote $X_\Gamma^{\circ}:=(T^*E)^{\circ}/\Gamma$,
therefore, the symplectic form $\omega_E$ descends uniquely to a holomorphic symplectic form 
$\omega_{X_\Gamma^{\circ}}$ on $X_\Gamma^{\circ}$ satisfying
$q_\Gamma^*\omega_{X_\Gamma^{\circ}}=\omega_E$.

\begin{lemma}
The singularities of $X_\Gamma$ are rational double points.  More precisely, at a point whose stabilizer has order $a$, the singularity is of type $A_{a-1}$.
\end{lemma}

\begin{proof}
At a fixed point, choose coordinates $(u,v)$ on the tangent space of $T^*E$ in which the stabilizer $\mu_a$ acts by
\[
        (u,v)\longmapsto (\zeta u,\zeta^{-1}v),
        \qquad \zeta=e^{2\pi i/a}.
\]
This is the standard cyclic subgroup of $\operatorname{SL}_2(\mbC)$ giving the Kleinian singularity of type $A_{a-1}$.  
Indeed, the invariant ring is generated by
\[
        U=u^a,\qquad V=v^a,
        \qquad W=uv,
\]
with the relation
\[
        UV=W^a.
\]
Thus the quotient is locally isomorphic to $\Spec \mbC[U,V,W]/(UV-W^a)$.
\end{proof}

For the five root types, the singularities are as following
\begin{table}[H]
\centering
\small
\renewcommand{\arraystretch}{1.2}
\begin{tabular}{c|c|c}
\textup{type} & $|\Gamma|$ & \textup{singularities of }$X_\Gamma$ \\
\hline
$\wt A_0$ & $1$ & none ($X_\Gamma=T^*E$ is smooth) \\
$\wt D_4$ & $2$ & $4A_1$ \\
$\wt E_6$ & $3$ & $3A_2$ \\
$\wt E_7$ & $4$ & $2A_3+A_1$ \\
$\wt E_8$ & $6$ & $A_5+A_2+A_1$
\end{tabular}
\caption{Singularities in the five families.}
\end{table}

Now we recall the facts about the uniqueness of crepant resolutions.

\begin{lemma}\label{lem:unique-crepant}
Let $X$ be a normal complex surface.  
If $X$ admits a crepant resolution, 
then its crepant resolution is unique up to the isomorphism over $X$.
\end{lemma}

\begin{proof}[Sketch of proof]
Let $f:Y\to X$ be crepant, in particular, $K_X$ is $\mathbb Q$-Cartier.  
Any divisor over the smooth surface $Y$ can be extracted by a sequence of blow-ups $g:Z\to Y$,
and we have
\[
 K_Z=(f\circ g)^*K_X+K_{Z/Y},
\]
where $K_{Z/Y}$ is effective.  
Hence all discrepancies of $X$ are nonnegative, so $X$ has canonical singularities. 
By the classification of canonical complex surface singularities,  $X$ has only Du Val singularities.
Thus we have that its minimal resolution $\mu:X_{\min}\to X$ is crepant.

Since every resolution of a normal surface factors through $X_{\min}$, we can write $f=\mu\circ h$.  
If $h:Y\to X_{\min}$ were not an isomorphism, 
then we have
\[
 K_Y=h^*K_{X_{\min}}+K_{Y/X_{\min}}
    =f^*K_X+K_{Y/X_{\min}},
\]
where $K_{Y/X_{\min}}$ is nonzero and effective, contradicting to the assumption that$K_Y=f^*K_X$.  
Thus every crepant resolution is isomorphic to $X_{\min}$ over $X$.
\end{proof}

\begin{corollary}
The morphism
\[
        \pi_\Gamma:S_\Gamma=\Gamma\text{-}\Hilb(T^*E)\longrightarrow X_\Gamma
\]
is the unique crepant resolution of $X_\Gamma$.
\end{corollary}

\begin{proof}
For $|\Gamma|=1$, one has $X_1=S_1=T^*E$ and $\pi_1$ is the identity.
Suppose that $|\Gamma|>1$.
For a finite subgroup of $\operatorname{SL}_2(\mbC)$, 
we have the fact that the $\Gamma$-Hilbert scheme is the minimal resolution of the quotient surface.

Since the minimal resolution of a Du Val surface singularity is crepant, $\pi_\Gamma$ is crepant.  
Therefore, Lemma \ref{lem:unique-crepant} gives uniqueness.
\end{proof}

\begin{definition}
Let $\omega_{S_\Gamma}$ be the unique regular two-form on $S_\Gamma$ whose restriction over $X_\Gamma^{\circ}$ equals $\pi_\Gamma^*\omega_{X_\Gamma^{\circ}}$.
\end{definition}

The form is non-degenerate since $\pi_\Gamma$ is crepant, thus $S_\Gamma$ is a holomorphic symplectic surface.
\subsection{Stability of the parabolic Higgs bundles}\label{sec:parabolic_form}

For the trivial group and the four quotient groups, respectively, we denote
\[
        (C,D)=
        \begin{cases}
        (E,\{o\}),&|\Gamma|=1,\\
        (E/\Gamma,\text{the reduced branch divisor of }p_\Gamma),
        &|\Gamma|>1.
        \end{cases}
\]
A parabolic bundle $V_*$ consists of a vector bundle $V$ on $C$,
and for each marked point $p\in D$, 
a flag
\[
        V_p=F_{p,0}\supset F_{p,1}\supset\cdots\supset F_{p,\ell_p}=0.
\]
We write $\ParEnd(V_*)\subset \End(V)$
for endomorphisms preserving every flag, 
and $\SParEnd(V_*)\subset \End(V)$ for strongly parabolic endomorphisms, 
i.e. for any endomorphism $f \in \SParEnd(V_*)$,
the induced fiber endomorphism $f_p\colon V_p\to V_p$ satisfies
\[
        f_p(F_{p,i})\subset F_{p,i+1}
        \qquad \text{for every }i.
\]
A parabolic Higgs field in such case is a section 
\[
        \theta\in H^0\bigl(C,\SParEnd(V_*)\otimes K_C(D)\bigr).
\]
i.e. $\theta$ is a logarithmic Higgs field whose residue at every marked point strictly lowers the parabolic flag.

We use the convention of the stability of parabolic Higgs bundles in \cite{Groechenig}. 
Let
\[
        0\leq \alpha_{p,0}<\alpha_{p,1}<\cdots
        <\alpha_{p,\ell_p-1}<1
\]
be the weights attached to
\[
        V_p=F_{p,0}\supset F_{p,1}\supset\cdots
        \supset F_{p,\ell_p}=0 .
\]
For a nonzero subbundle $F\subset V$, put
\[
        F_{*,p,i}:=F_p\cap F_{p,i},
        \qquad
        m_{p,i}(F):=
        \dim\frac{F_p\cap F_{p,i}}{F_p\cap F_{p,i+1}} .
\]
\begin{definition}
We define the parabolic degree and slope as
\[
        \pardeg(F_*)=
        \deg F+\sum_{p\in D}\sum_{i=0}^{\ell_p-1}
        \alpha_{p,i}\,m_{p,i}(F),
        \qquad
        \mu_{\mathrm{par}}(F_*)=\frac{\pardeg(F_*)}{\rk F}.
\]
A strongly parabolic Higgs bundle $(V_*,\theta)$ is stable if $\mu_{\mathrm{par}}(F_*)<\mu_{\mathrm{par}}(V_*)$
for every subbundle $0\neq F\subsetneq V$ satisfying $\theta(F)\subset F\otimes K_C(D)$.
\end{definition}

In particular, for $|\Gamma|=1$, we take $(C,D)=(E,\{o\})$ and consider rank $m$, degree $0$ strongly parabolic Higgs bundles with flag
\[
        V_o\supset\ell\supset0,
        \qquad \dim\ell=1.
\]
Assign weight $0$ to $V_o/\ell$ and weight $\varepsilon\in(0,1)$ to $\ell$.  
Then
\[
 \pardeg(V_*)=\varepsilon,\qquad
 \mu_{\mathrm{par}}(V_*)=\frac{\varepsilon}{m},
\]
while a $\theta$-invariant subbundle $0\ne F\subsetneq V$ satisfies
\[
 \pardeg(F_*)=
 \begin{cases}
  \deg F+\varepsilon,&\ell\subset F_o,\\
  \deg F,&\ell\not\subset F_o.
 \end{cases}
\]
Thus the $\wt A_0$ stability condition is
\begin{equation}\label{eq:A0-stability}
 \frac{\deg F+\varepsilon\,\mathbf{1}_{\{\ell\subset F_o\}}}{\rk F}
 <\frac{\varepsilon}{m}
\end{equation}
for every proper nonzero $\theta$-invariant subbundle $F$. 

\begin{remark}\label{rem:marked-parabolic}
The above parabolic structure can be identified with the marked structure which is introduced in \cite{GNR}, 
that is, triples $(V,\theta,v)$ in which $(V,\theta)$ is semistable of rank $m$, degree $0$ and $0\ne v\in\ell$.
The stability condition is given as the following:
\begin{equation}
 \text{There is no proper $\theta$-invariant subbundle $F\subsetneq V$ satisfies
 $\mu(F)\geq0$ and $v\in F_o$.}
\end{equation}
On the locus where $(V,\theta)$ is a sum of pairwise distinct degree $0$ Higgs line bundles, 
the condition says that $v$ has a nonzero component in every summand.
\end{remark}

For each of the four nontrivial quotient cases, we begin with the parabolic data corresponding to the surface $S_\Gamma$.  
At any torsion point $p$ of stabilizer order $a$, we write
\[
 \lambda_a(1)=
 \left(|\Gamma|,|\Gamma|-\frac{|\Gamma|}{a},
 |\Gamma|-\frac{2|\Gamma|}{a},\ldots,\frac{|\Gamma|}{a},0\right),
 \qquad
 \mathbf w_a=
 \left(0,\frac1a,\ldots,\frac{a-1}{a}\right).
\]
Here $\lambda_a(1)$ denotes the flag
\[
 V_p=F_{p,0}\supset\cdots\supset F_{p,a}=0,
 \qquad
 \dim F_{p,j}=|\Gamma|-\frac{j|\Gamma|}{a},
\]
and $\mathbf w_a$ assigns weight $j/a$ to the quotient $F_{p,j}/F_{p,j+1}$, 
which has dimension $|\Gamma|/a$.

The following quivers record these flag dimensions:
the central vertex has rank $|\Gamma|$, and each arm records the nonzero proper dimensions in the flag at one torsion point.
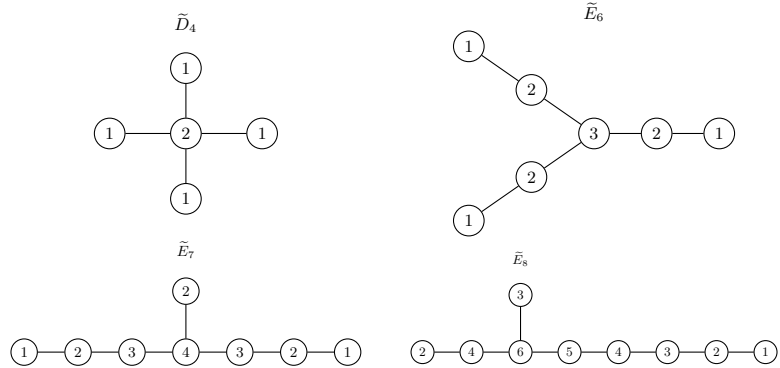
\begin{center}
\scriptsize
\setlength{\tabcolsep}{6pt}
\renewcommand{\arraystretch}{1.7}
\begin{tabular}{c@{\qquad}c}
\begin{tikzpicture}[
        baseline=-0.5ex,
        scale=0.72,
        every node/.style={transform shape},
        v/.style={circle,draw,inner sep=1pt,minimum size=16pt}
]
        \node[v] (c) at (0,0) {$2$};
        \node[v] (l) at (-1.4,0) {$1$};
        \node[v] (u) at (0,1.2) {$1$};
        \node[v] (d) at (0,-1.2) {$1$};
        \node[v] (r) at (1.4,0) {$1$};
        \draw (c)--(l) (c)--(u) (c)--(d) (c)--(r);
        \node[above] at (0,1.75) {$\wt D_4$};
\end{tikzpicture}
&
\begin{tikzpicture}[
        baseline=-0.5ex,
        scale=0.72,
        every node/.style={transform shape},
        v/.style={circle,draw,inner sep=1pt,minimum size=16pt}
]
        \node[v] (c) at (0,0) {$3$};
        \node[v] (ul1) at (-1.15,0.8) {$2$};
        \node[v] (ul2) at (-2.3,1.6) {$1$};
        \node[v] (dl1) at (-1.15,-0.8) {$2$};
        \node[v] (dl2) at (-2.3,-1.6) {$1$};
        \node[v] (r1) at (1.15,0) {$2$};
        \node[v] (r2) at (2.3,0) {$1$};
        \draw (c)--(ul1)--(ul2) (c)--(dl1)--(dl2) (c)--(r1)--(r2);
        \node[above] at (0,1.95) {$\wt E_6$};
\end{tikzpicture}
\\
\begin{tikzpicture}[
        baseline=-0.5ex,
        scale=0.62,
        every node/.style={transform shape},
        v/.style={circle,draw,inner sep=1pt,minimum size=16pt}
]
        \node[v] (c) at (0,0) {$4$};
        \node[v] (u) at (0,1.3) {$2$};
        \node[v] (l1) at (-1.15,0) {$3$};
        \node[v] (l2) at (-2.3,0) {$2$};
        \node[v] (l3) at (-3.45,0) {$1$};
        \node[v] (r1) at (1.15,0) {$3$};
        \node[v] (r2) at (2.3,0) {$2$};
        \node[v] (r3) at (3.45,0) {$1$};
        \draw (c)--(u) (c)--(l1)--(l2)--(l3)
              (c)--(r1)--(r2)--(r3);
        \node[above] at (0,1.9) {$\wt E_7$};
\end{tikzpicture}
&
\begin{tikzpicture}[
        baseline=-0.5ex,
        scale=0.54,
        every node/.style={transform shape},
        v/.style={circle,draw,inner sep=1pt,minimum size=16pt}
]
        \node[v] (c) at (0,0) {$6$};
        \node[v] (u) at (0,1.4) {$3$};
        \node[v] (l1) at (-1.2,0) {$4$};
        \node[v] (l2) at (-2.4,0) {$2$};
        \node[v] (r1) at (1.2,0) {$5$};
        \node[v] (r2) at (2.4,0) {$4$};
        \node[v] (r3) at (3.6,0) {$3$};
        \node[v] (r4) at (4.8,0) {$2$};
        \node[v] (r5) at (6.0,0) {$1$};
        \draw (c)--(u) (c)--(l1)--(l2)
              (c)--(r1)--(r2)--(r3)--(r4)--(r5);
        \node[above] at (0,2.0) {$\wt E_8$};
\end{tikzpicture}
\end{tabular}
\end{center}

Now, as for the four nontrivial quotient cases with parabolic data corresponding to the Hilbert scheme $S_\Gamma^{[m]}$.
We first put $N=|\Gamma|m$, for $0 \neq p \in D$, let $a_p$ be the order of its stabilizer.
For every order $a$, we write
\[
 \lambda_a(m)=
 \left(N,N-\frac Na,N-\frac{2N}{a},\ldots,\frac Na,0\right),
 \qquad
 \mathbf w_a=
 \left(0,\frac1a,\ldots,\frac{a-1}{a}\right).
\]
Here $\lambda_a(m)$ denotes the flag
\begin{equation}
 V_p=F_{p,0}\supset\cdots\supset F_{p,a}=0,
 \qquad
 \dim F_{p,j}=N-\frac{jN}{a},
\end{equation}
and $\mathbf w_a$ assigns weight $j/a$ to the quotient $F_{p,j}/F_{p,j+1}$, which has dimension $N/a$.

At $0=p_\Gamma(o)$, we denote
\[
 \lambda_{|\Gamma|}^+(m)=(N,N-m,\ldots,m,1,0),
 \qquad
 \mathbf w_{|\Gamma|}^+
 =\left(0,\frac1{|\Gamma|},\ldots,
 \frac{|\Gamma|-1}{|\Gamma|},\varepsilon\right).
\]
Here
\[
 1>\varepsilon>
 \max\left\{\frac{|\Gamma|-1}{|\Gamma|},\frac{m-1}{m}\right\}.
\]

The following quivers record the flag dimensions; the central vertex
has rank $N$, and $1_\ast$ denotes the line $\ell$.
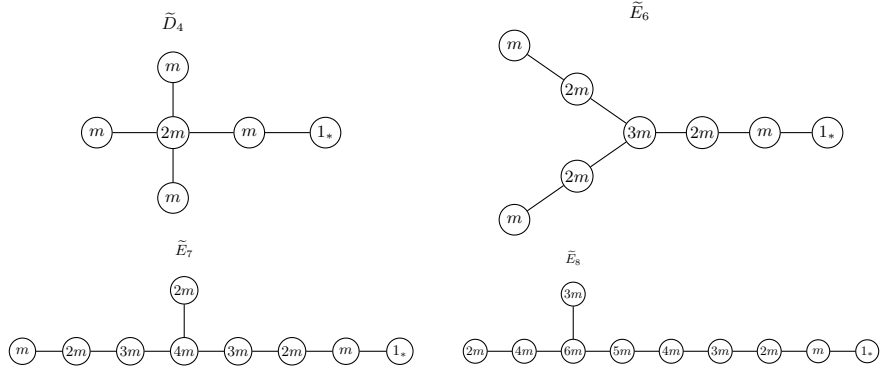
\begin{center}
\scriptsize
\setlength{\tabcolsep}{6pt}
\renewcommand{\arraystretch}{1.7}
\begin{tabular}{c@{\qquad}c}
\begin{tikzpicture}[
        baseline=-0.5ex,
        scale=0.72,
        every node/.style={transform shape},
        v/.style={circle,draw,inner sep=1pt,minimum size=16pt}
]
        \node[v] (c) at (0,0) {$2m$};
        \node[v] (l) at (-1.4,0) {$m$};
        \node[v] (u) at (0,1.2) {$m$};
        \node[v] (d) at (0,-1.2) {$m$};
        \node[v] (r1) at (1.4,0) {$m$};
        \node[v] (r2) at (2.8,0) {$1_\ast$};
        \draw (c)--(l) (c)--(u) (c)--(d) (c)--(r1)--(r2);
        \node[above] at (0,1.75) {$\wt D_4$};
\end{tikzpicture}
&
\begin{tikzpicture}[
        baseline=-0.5ex,
        scale=0.72,
        every node/.style={transform shape},
        v/.style={circle,draw,inner sep=1pt,minimum size=16pt}
]
        \node[v] (c) at (0,0) {$3m$};
        \node[v] (ul1) at (-1.15,0.8) {$2m$};
        \node[v] (ul2) at (-2.3,1.6) {$m$};
        \node[v] (dl1) at (-1.15,-0.8) {$2m$};
        \node[v] (dl2) at (-2.3,-1.6) {$m$};
        \node[v] (r1) at (1.15,0) {$2m$};
        \node[v] (r2) at (2.3,0) {$m$};
        \node[v] (r3) at (3.45,0) {$1_\ast$};
        \draw (c)--(ul1)--(ul2) (c)--(dl1)--(dl2) (c)--(r1)--(r2)--(r3);
        \node[above] at (0,1.95) {$\wt E_6$};
\end{tikzpicture}
\\
\begin{tikzpicture}[
        baseline=-0.5ex,
        scale=0.62,
        every node/.style={transform shape},
        v/.style={circle,draw,inner sep=1pt,minimum size=16pt}
]
        \node[v] (c) at (0,0) {$4m$};
        \node[v] (u) at (0,1.3) {$2m$};
        \node[v] (l1) at (-1.15,0) {$3m$};
        \node[v] (l2) at (-2.3,0) {$2m$};
        \node[v] (l3) at (-3.45,0) {$m$};
        \node[v] (r1) at (1.15,0) {$3m$};
        \node[v] (r2) at (2.3,0) {$2m$};
        \node[v] (r3) at (3.45,0) {$m$};
        \node[v] (r4) at (4.6,0) {$1_\ast$};
        \draw (c)--(u) (c)--(l1)--(l2)--(l3) (c)--(r1)--(r2)--(r3)--(r4);
        \node[above] at (0,1.9) {$\wt E_7$};
\end{tikzpicture}
&
\begin{tikzpicture}[
        baseline=-0.5ex,
        scale=0.54,
        every node/.style={transform shape},
        v/.style={circle,draw,inner sep=1pt,minimum size=16pt}
]
        \node[v] (c) at (0,0) {$6m$};
        \node[v] (u) at (0,1.4) {$3m$};
        \node[v] (l1) at (-1.2,0) {$4m$};
        \node[v] (l2) at (-2.4,0) {$2m$};
        \node[v] (r1) at (1.2,0) {$5m$};
        \node[v] (r2) at (2.4,0) {$4m$};
        \node[v] (r3) at (3.6,0) {$3m$};
        \node[v] (r4) at (4.8,0) {$2m$};
        \node[v] (r5) at (6.0,0) {$m$};
        \node[v] (r6) at (7.2,0) {$1_\ast$};
        \draw (c)--(u) (c)--(l1)--(l2)
              (c)--(r1)--(r2)--(r3)--(r4)--(r5)--(r6);
        \node[above] at (0,2.0) {$\wt E_8$};
\end{tikzpicture}
\end{tabular}
\end{center}

With respect to those flags, weights and the stability condition, 
we give the following definition of the moduli space of stable strongly parabolic Higgs bundles.

\begin{definition}[\cite{Groechenig}]
Let $N=|\Gamma|m$,
we denote by $\mcM_{\Gamma,m}$ the moduli space of
isomorphism classes of stable strongly parabolic Higgs bundles
$(V_*,\theta)$ on $(C,D)$, where
\[
 \theta\in H^0\bigl(C,\SParEnd(V_*)\otimes K_C(D)\bigr),
\]
with the following parabolic data:

If $|\Gamma|=1$, then
\[
 (\rk V,\deg V)=(m,0),
 \qquad
 V_o\supset\ell\supset0,
 \qquad
 \dim\ell=1,
\]
and the weights are $(0,\varepsilon)$.

If $|\Gamma|>1$, then
\[
 (\rk V,\deg V)=(N,-N).
\]
For $p\in D$, let $a_p$ be the order of the stabilizer of a point over $p$.  
When $m=1$, the flag and weights at every $p\in D$ are
$(\lambda_{a_p}(1),\mathbf w_{a_p})$.  
When $m>1$, they are given as the following
\[
 (\lambda_{a_p}(m),\mathbf w_{a_p})
 \quad\text{at }p\ne0,
 \qquad
 (\lambda_{|\Gamma|}^+(m),\mathbf w_{|\Gamma|}^+)
 \quad\text{at }0=p_\Gamma(o).
\]
\end{definition}

\subsection{Symplectic form of the parabolic Higgs moduli}

Following the infinitesimal deformation theory of Hitchin pairs developed by Biswas and Ramanan \cite{BR},
we can construct the symplectic form on the moduli space.

For arbitrary $a\in\ParEnd(V_*)$, together with a strongly parabolic Higgs field $\theta\in H^0\bigl(C,\SParEnd(V_*)\otimes K_C(D)\bigr)$,
we first write $[\theta,a]:=\theta\circ a-(a\otimes 1)\circ\theta$.

\begin{definition}
Let $(V_*,\theta)$ be a stable strongly parabolic Higgs bundle on $(C,D)$.
Let us denote
\[
 \mathcal P:=\ParEnd(V_*),
 \qquad
 \mathcal Q:=\SParEnd(V_*)\otimes K_C(D).
\]
then we define its deformation complex to be the following two-term complex
\begin{equation}\label{eq:parabolic-deformation-complex}
 \mcK^\bullet_{(V_*,\theta)}
 :=
 \mathcal P\xrightarrow{[\theta,\,\cdot\,]}\mathcal Q.
\end{equation}
Indeed, the differential is well defined since the elements in
$\mathcal P$ preserves the parabolic flags, whereas $\theta$ strictly lowers them.
\end{definition}

Next, we tend to compute the hypercohomology of this complex explicitly.  
Choose distinct points $x_0,x_1\in C\setminus D$ and set
\[
        U_0:=C\setminus\{x_0\},
        \qquad
        U_1:=C\setminus\{x_1\},
        \qquad
        \mathfrak U:=\{U_0,U_1\}.
\]
The opens $U_0,U_1$, and $U_{01}:=U_0\cap U_1$ are affine.  Thus
$\mathfrak U$ is a Leray cover for the coherent sheaves $\mathcal P$
and $\mathcal Q$.  For any coherent sheaf $\mathcal F$ it has
\[
 \check C^0(\mathfrak U,\mathcal F)
 =\mathcal F(U_0)\oplus\mathcal F(U_1),
 \qquad
 \check C^1(\mathfrak U,\mathcal F)=\mathcal F(U_{01}),
 \qquad
 \check C^p(\mathfrak U,\mathcal F)=0\quad(p\geq2).
\]
Hence we have the following commutative diagram
\[
\begin{tikzcd}[column sep=large,row sep=large]
 \check C^0(\mathfrak U,\mathcal P)
 \arrow[r,"\delta"]
 \arrow[d,"d_\theta"'] &
 \check C^1(\mathfrak U,\mathcal P)
 \arrow[d,"d_\theta"] \\
 \check C^0(\mathfrak U,\mathcal Q)
 \arrow[r,"\delta"'] &
 \check C^1(\mathfrak U,\mathcal Q).
\end{tikzcd}
\]
Using the differential $D^p \coloneqq \delta+(-1)^p d_\theta$ of degree $p$, 
we obtain a sequence of \v Cech cochains
\[
 0\longrightarrow
 \check C^0(\mathfrak U,\mathcal P)
 \xrightarrow{\ D^0\ }
 \check C^1(\mathfrak U,\mathcal P)
 \oplus\check C^0(\mathfrak U,\mathcal Q)
 \xrightarrow{\ D^1\ }
 \check C^1(\mathfrak U,\mathcal Q)
 \longrightarrow0,
\]
where
\[
 \begin{aligned}
 D^0(u_0,u_1)
 &=
 \bigl(u_1-u_0,[\theta,u_0],[\theta,u_1]\bigr),\\
 D^1(\alpha_{01},\beta_0,\beta_1)
 &=
 \beta_1-\beta_0-[\theta,\alpha_{01}].
 \end{aligned}
\]
Therefore, we can write the first hypercohomology as 
\begin{equation}
 \mathbb H^1(C,\mcK^\bullet_{(V_*,\theta)})
 =
 \frac{
 \ker\!\left(
 D^1:
 \check C^1(\mathfrak U,\mathcal P)
 \oplus\check C^0(\mathfrak U,\mathcal Q)
 \longrightarrow
 \check C^1(\mathfrak U,\mathcal Q)
 \right)}
 {\operatorname{im}\!\left(
 D^0:
 \check C^0(\mathfrak U,\mathcal P)
 \longrightarrow
 \check C^1(\mathfrak U,\mathcal P)
 \oplus\check C^0(\mathfrak U,\mathcal Q)
 \right)} .
\end{equation}

\begin{theorem}[\cite{BR}]
The tangent space at $(V_*,\theta) \in \mcM$ is naturally identified with
the first hypercohomology of its deformation complex:
\[
 T_{(V_*,\theta)}\mcM
 \simeq\mathbb H^1\bigl(C,\mcK^\bullet_{(V_*,\theta)}\bigr).
\]
\end{theorem}

For the remaining part of this section, 
we will construct the symplectic form on $\mcM$ explicitly.
First note that by taking the trace, we can obtain a sheaf pairing
\begin{equation}\label{eq:parabolic-pairing}
 \mathcal P\otimes\mathcal Q\longrightarrow K_C,
 \qquad
 \langle a,b\otimes\eta\rangle_{\mathrm{tr}}=\tr(ab)\eta,
\end{equation}
and one can show that the pairing is perfect, i.e.
\[
 \mathcal P^\vee
 \simeq\SParEnd(V_*)\otimes\mathcal O_C(D),
 \qquad
 \mathcal Q\simeq\mathcal P^\vee\otimes K_C.
\]

On the other hand, the following map can be shown to be well-defined
\begin{equation}
 \pi:\mathbb H^1(C,\mcK^\bullet_{(V_*,\theta)})
 \longrightarrow H^1(C,\mathcal P), 
 \quad \pi\bigl([\alpha_{01},\beta_0,\beta_1]\bigr)
 =[\alpha_{01}].
\end{equation}
We now construct the cohomology pairing from \eqref{eq:parabolic-pairing}.

With respect to the previous Leray cover, 
we let $[\alpha_{01}]\in H^1(C,\mathcal P)$, 
and represent $\gamma\in H^0(C,\mathcal Q)$ by the \v Cech $0$-cocycle
$(\gamma_0,\gamma_1) \in \check C^0(\mathfrak U,\mathcal Q)$.
Take $\tr(\alpha_{01}\gamma_1)\in K_C(U_{01})$,
since $\check C^2(\mathfrak U,K_C)=0$, $\tr(\alpha_{01}\gamma_1)$ is a cocycle.  
If $\alpha_{01}$ is replaced by the cocycle $\alpha_{01}+s_1-s_0$, 
where $s_i\in\mathcal P(U_i)$, then we have
\[
\begin{aligned}
 \tr\bigl((\alpha_{01}+s_1-s_0)\gamma_1\bigr)
 -\tr(\alpha_{01}\gamma_1)
 &=\tr(s_1\gamma_1)-\tr(s_0\gamma_0)\\
 &=\bigl(\delta(\tr(s_i\gamma_i))\bigr)_{01}.
\end{aligned}
\]
Thus the class $[\tr(\alpha_{01}\gamma_1)]\in H^1(C,K_C)$ depends only
on $[\alpha_{01}]$ and $\gamma$, and the sheaf pairing \eqref{eq:parabolic-pairing} induces
\[
 H^1(C,\mathcal P)\otimes H^0(C,\mathcal Q)
 \longrightarrow H^1(C,K_C),\qquad
 ([\alpha_{01}],\gamma)
 \longmapsto[\tr(\alpha_{01}\gamma_1)].
\]
Composing with the trace map gives rise to the following
\[
 H^1(C,\mathcal P)\otimes H^0(C,\mathcal Q)
 \xrightarrow{\ \cup_{\mathrm{tr}}\ } H^1(C,K_C)
 \xrightarrow{\ \operatorname{Tr}_C\ } \mbC.
\]
Since $\mathcal Q\simeq\mathcal P^\vee\otimes K_C$, 
by Serre duality, this pairing is also perfect.
In this way, we finally have the following 
definition of the one-form on the moduli space $\mcM$.

\begin{definition}
Taking $\gamma=\theta$ and $[\alpha_{01}]=\pi(u)$, we define the
one-form $\Theta_{\mcM}$ by
\[
 \Theta_{\mcM,(V_*,\theta)}(u)
 =
 \operatorname{Tr}_C\bigl(\pi(u)\cup_{\mathrm{tr}}\theta\bigr).
\]
Explicitly, for $u=[\alpha_{01},\beta_0,\beta_1]$, this can be written as
\begin{equation}\label{eq:parabolic-one-form}
 \Theta_{\mcM,(V_*,\theta)}(u)
 =
 \operatorname{Tr}_C
 \bigl([\tr(\alpha_{01}\theta_1)]\bigr),
 \qquad \theta_1=\theta|_{U_1}.
\end{equation}
\end{definition}

Biswas and Mukherjee \cite[\S2.2]{BM} identify
$d\Theta_{\mcM}$ with the holomorphic symplectic form,
i.e. the two-form $\Omega_{\mcM}:=d\Theta_{\mcM}$ is nondegenerate.
Indeed, for $u=[\alpha_{01},\beta_0,\beta_1]$ and $v=[\alpha'_{01},\beta'_0,\beta'_1]$, 
we can calculate the symplectic form as
\begin{equation}
 \Omega_{\mcM}(u,v)
 =
 \operatorname{Tr}_C
 \left(
 \left[
 \tr\bigl(\alpha'_{01}\beta_0-\alpha_{01}\beta'_1\bigr)
 \right]
 \right).
\end{equation}

\section{Proof of the main theorem}

In this section, we will reach our main goal:
To compare the natural holomorphic symplectic form on the Hilbert scheme of points of $S_\Gamma$, 
and the symplectic form obtained from the deformation complex on the corresponding parabolic Higgs moduli space.

We first treat the trivial group case which first constructed in \cite{GNR}, 
and then extend such result to the four nontrivial quotient cases.

In each case, we do it by comparing the forms on an open dense reduced locus 
and extend the comparisons globally.

\subsection{The trivial-group case}

When $|\Gamma|=1$, the quotient is the elliptic curve itself,
in this case, we take the divisor $D=\{o\}$.  
Let $\mcM_{1,m}$ be the stable parabolic Higgs moduli space on $E$ with the flag
\[
        V_o\supset\ell\supset0,
        \qquad \dim\ell=1,
\]
weights $0$ and $\varepsilon$, and stability condition
\eqref{eq:A0-stability}.
The deformation complex and trace pairing construction of Section \ref{sec:parabolic_form} thus apply to $(C,D)=(E,\{o\})$.

\begin{proposition}\label{prop:trivial-group-symplectomorphism}
Under the relative Fourier--Mukai isomorphism \cite[Theorem~5.1]{Groechenig},
\[
        \Phi_{1,m}:(T^*E)^{[m]}\xrightarrow{\sim}\mcM_{1,m},
\]
one has
\[
        \Phi_{1,m}^*\Omega_{\mcM_{1,m}}
        =\omega_{(T^*E)^{[m]}}.
\]
\end{proposition}

\begin{proof}
First, we recall that how the relative Fourier--Mukai transform $\mathsf F$ 
identifies the open dense reduced locus $U_m\subset (T^*E)^{[m]}$ 
with the corresponding open locus $\mcV_m\subset\mcM_{1,m}$.
Fix a nonzero differential $dz\in H^0(E,K_E)$.  
We identify $E$ with $E^\vee=\Pic^0(E)$ by
\[
 \mathrm{AJ}(x)=L_x:=\mcO_E(x-o),
\]
and we choose the Poincar\'e bundle $\mathcal P$ normalized by $\mathcal P|_{\{o\}\times E^\vee}\simeq\mcO_{E^\vee}$.  
From the Serre-duality, we have the identification
\begin{equation}\label{eq:jacobian-serre}
 \bigl\langle d\mathrm{AJ}_x(\dot x),\eta\bigr\rangle
 =\eta_x(\dot x),
 \qquad \eta\in H^0(E,K_E).
\end{equation}
A point of $T^*E$ will be written as $a=(x,\xi dz|_x)$.  
If $\lambda_E$ is the tautological one-form, then we have
\[
 \lambda_{E,a}(\dot x,\dot\xi)=\xi\,dz|_x(\dot x),
 \qquad \omega_E=d\lambda_E.
\]

It suffices first to work on the open dense reduced locus $U_m\subset (T^*E)^{[m]}$.  
Note that a point of such locus is an unordered tuple of distinct points
\[
 a_i=(x_i,\xi_i dz),\qquad 1\leq i\leq m.
\]
The structure sheaf of the corresponding reduced subscheme is $\bigoplus_i\mcO_{a_i}$.  
Additivity of the relative Fourier--Mukai transform and the BNR correspondence (see \cite{BNR}) therefore give
\begin{equation}\label{eq:trivial-group-splitting}
 (V,\theta)=\bigoplus_{i=1}^m(L_{x_i},\xi_i dz),
 \qquad
 \theta=\diag(\xi_1dz,\ldots,\xi_mdz).
\end{equation}
We also record the marked structure appeared in the transform.  
Let
\[
 s:\mcO_{T^*E}\twoheadrightarrow\bigoplus_i\mcO_{a_i}
\]
be the canonical quotient map and write $s=(s_1,\ldots,s_m)$. 
The relative Fourier--Mukai functor $\mathsf F$ identifies
\begin{equation}
 \Hom\bigl(\mathsf F(\mcO_{T^*E}),
            \mathsf F(\textstyle\bigoplus_i\mcO_{a_i})\bigr)
 \simeq V_o.
\end{equation}
Indeed, $\mathsf F(\mcO_{T^*E})\simeq \mcO_{\{o\}\times\mbA^1}[-1]$, 
and after the BNR correspondence the left hand side becomes
\[
 \Ext_E^1(\mbC_o,V)
 \simeq\Hom_E(V,\mbC_o)^*\simeq V_o.
\]
Under this identification,
$\mathsf F(s_i)$ gives a nonzero vector $v_i\in(L_{x_i})_o$, 
since $s_i$ is nonzero and $\mathsf F$ is an equivalence.  
Thus we obtain the parabolic/marked line as
\begin{equation}
 \ell=\mbC(v_1+\cdots+v_m)\subset V_o,
 \qquad v_i\ne0.
\end{equation}
Let $\mathcal N_m^{\circ}$ be the open locus in the coarse moduli space of semistable rank $m$, 
degree $0$ Higgs bundles consisting of direct sums
\[
 \bigoplus_{i=1}^m(L_{x_i},\xi_i dz)
\]
whose rank one Higgs summands are pairwise nonisomorphic, 
and let $\mcV_m\subset\mcM_{1,m}$ be the image of $U_m$.  
Forgetting the marking defines a map $q:\mcV_m\rightarrow\mathcal N_m^{\circ}$,
this is canonically an isomorphism.  
Indeed, let $Z\in U_m$ be the unordered set of its $m$ distinct points.  For
$a=(x(a),\xi(a)dz)\in Z$, let us denote
\[
 H_a:=(L_{x(a)},\xi(a)dz),
 \qquad
 H_Z:=\bigoplus_{a\in Z}H_a.
\]
For $a,b\in Z$, we can compute the hom set as
\[
\begin{aligned}
 \Hom_{\mathrm{Higgs}}(H_a,H_b)
 &=
 \left\{
 f\in H^0\bigl(E,L_{x(b)}\otimes L_{x(a)}^{-1}\bigr):
 (\xi(b)-\xi(a))f\,dz=0
 \right\}                                          \\
 &=
 \begin{cases}
  \mbC,&a=b,\\
  0,&a\ne b.
 \end{cases}
\end{aligned}
\]
Since $Z$ consists of distinct points, we have
\[
 \Aut_{\mathrm{Higgs}}(H_Z)
 =
 (\mbC^*)^Z
 \coloneqq
 \{t:Z\to\mbC^*\},
 \qquad
 g_t=\bigoplus_{a\in Z}t(a)\,\mathrm{id}_{H_a}.
\]

We write $v_a\in(L_{x(a)})_o\setminus\{0\}$ for the vector obtained from
the quotient $\mcO_{T^*E}\twoheadrightarrow\mcO_a$.  
By Remark \ref{rem:marked-parabolic} and the above calculations,
the marked lines of $H_Z$ are exactly of the following form
\[
 \ell_c
 =
 \mbC\!\left(\sum_{a\in Z}c(a)v_a\right),
 \qquad c\in(\mbC^*)^Z,
\]
so we can see that the automorphisms act on the marked line by simply the scalar multiplications.
Therefore, one can conclude that the relative Fourier--Mukai morphism restricted to the reduced locus gives an isomorphism
\[
 \psi:U_m\xrightarrow{\sim}\mathcal N_m^\circ,
 \qquad
 \psi(Z)=[H_Z],
 \qquad
 \psi^{-1}([H_Z])=Z.
\]

Second, we will compare the parabolic and ordinary semistable Higgs deformation
theories, write
\[
 \mcK_{\mathrm{par}}^\bullet=
 [\ParEnd(V_*)\xrightarrow{[\theta,\cdot]}
  \SParEnd(V_*)\otimes K_E(o)],
 \qquad
 \mcK_{\mathrm{ss}}^\bullet=[\End(V)\xrightarrow{[\theta,\cdot]}
 \End(V)\otimes K_E].
\]
A regular $\End(V)$-valued one-form, 
regarded as a logarithmic one-form at $o$, has zero residue and is strongly parabolic,
hence we have $\End(V)\otimes K_E\subset \SParEnd(V_*)\otimes K_E(o)$.

Let $i_o:\{o\}\hookrightarrow E$ be the inclusion, 
choose a complement $W$ to $\ell$ in $V_o$, and write $V_o=\ell\oplus W$.  
With respect to this decomposition, an endomorphism of $V_o$ has the form
\[
 f_o=\begin{pmatrix}a&b\\ c&d\end{pmatrix},
 \qquad
 c\in\Hom(\ell,W).
\]
It preserves $\ell$ precisely when $c=0$.
Thus, if we set
\[
 A_o:=\Hom(\ell,V_o/\ell),
\]
then the quotient $\End(V)/\ParEnd(V_*)$ is the skyscraper sheaf $i_{o*}A_o$.

Similarly, we choose a local coordinate $t$ centered at $o$, 
then a section of $\SParEnd(V_*)\otimes K_E(o)$ can be written locally as
\[
 R(t)\frac{dt}{t}.
\]
From the strongly parabolic condition, 
its residue $R(0)$ must strictly lower the flag $V_o\supset\ell\supset0$, 
i.e. we have $R(0)(V_o)\subset\ell,\ R(0)(\ell)=0$.
Again, in the decomposition $V_o=\ell\oplus W$, this means
\[
 R(0)=\begin{pmatrix}0&b\\0&0\end{pmatrix},
 \qquad b\in\Hom(W,\ell)
       \simeq\Hom(V_o/\ell,\ell).
\]
If we write $B_o:=\Hom(V_o/\ell,\ell)$,
then from the above calculations, we can obtain two exact sequences
\begin{equation}
 \begin{gathered}
 0\longrightarrow\ParEnd(V_*)\longrightarrow\End(V)
 \longrightarrow i_{o*}A_o\longrightarrow0,\\
 0\longrightarrow\End(V)\otimes K_E
 \longrightarrow\SParEnd(V_*)\otimes K_E(o)
 \xrightarrow{\operatorname{Res}_o}i_{o*}B_o\longrightarrow0.
 \end{gathered}
\end{equation}

Now, we introduce the following ``common'' complex
\[
 \mathcal D^\bullet=
 [\End(V)\xrightarrow{[\theta,\cdot]}
  \SParEnd(V_*)\otimes K_E(o)].
\]
Define
\[
 A^\bullet:=[i_{o*}A_o\longrightarrow0], \quad B^\bullet:=[0\longrightarrow i_{o*}B_o].
\]
We thus have two short exact sequences of complexes
\begin{equation}\label{eq:comparison-of-exact-sequences}
 \begin{gathered}
 0\longrightarrow\mcK_{\mathrm{par}}^\bullet
 \longrightarrow\mathcal D^\bullet
 \longrightarrow A^\bullet\longrightarrow0,\\
 0\longrightarrow\mcK_{\mathrm{ss}}^\bullet
 \longrightarrow\mathcal D^\bullet
 \longrightarrow B^\bullet\longrightarrow0,
 \end{gathered}
\end{equation}

For the above two short exact sequences, 
we will compute the induced long exact hypercohomology sequences.  
Since $A^\bullet=[i_{o*}A_o\to0]$, 
one has $\mathbb H^0(E,A^\bullet)=A_o$ and $\mathbb H^q(E,A^\bullet)=0$ for $q>0$.  
Hence the first short exact sequence in \eqref{eq:comparison-of-exact-sequences} induces the long exact sequence as the following
\begin{equation}
 \begin{aligned}
 0\longrightarrow
 \mathbb H^0(E,\mcK_{\mathrm{par}}^\bullet)
 &\longrightarrow \mathbb H^0(E,\mathcal D^\bullet)
 \xrightarrow{\ \epsilon_o\ } A_o\\
 &\longrightarrow \mathbb H^1(E,\mcK_{\mathrm{par}}^\bullet)
 \longrightarrow \mathbb H^1(E,\mathcal D^\bullet)
 \longrightarrow0.
 \end{aligned}
\end{equation}

Since the Higgs summands are pairwise nonisomorphic, we have
\[
 \mathbb H^0(E,\mcK_{\mathrm{ss}}^\bullet)
 =\mathbb H^0(E,\mathcal D^\bullet)
 =\bigoplus_{i=1}^m\mbC\operatorname{id}_{L_{x_i}}.
\]
Under this identification, we can explicitly write the map $\epsilon_o$ as
\begin{equation}\label{eq:diagonal-automorphism-to-marking}
 (c_1,\ldots,c_m)\longmapsto
 \left[v_1+\cdots+v_m\longmapsto
       c_1v_1+\cdots+c_mv_m\right]\bmod\ell.
\end{equation}
Its kernel consists of the scalars of the form $c_1=\cdots=c_m$, and it is surjective.  
Therefore, we have
\begin{equation}\label{eq:parabolic-hypercohomology}
 \mathbb H^1(E,\mcK_{\mathrm{par}}^\bullet)
 \xrightarrow{\sim}\mathbb H^1(E,\mathcal D^\bullet).
\end{equation}
For the second sequence, since
$B^\bullet=[0\longrightarrow i_{o*}B_o]$, we have
\[
 \mathbb H^0(E,B^\bullet)=0,\qquad
 \mathbb H^1(E,B^\bullet)=B_o,\qquad
 \mathbb H^2(E,B^\bullet)=0.
\]
Thus the induced long exact sequence can be written as
\begin{equation}
 0\longrightarrow
 \mathbb H^1(E,\mcK_{\mathrm{ss}}^\bullet)
 \longrightarrow
 \mathbb H^1(E,\mathcal D^\bullet)
 \longrightarrow B_o
 \xrightarrow{\ \partial_B\ }
 \mathbb H^2(E,\mcK_{\mathrm{ss}}^\bullet).
\end{equation}
Again, we can calculate that $\partial_B$ is injective.
Thus, the long exact sequence induced by the second short exact sequence in
\eqref{eq:comparison-of-exact-sequences} now gives
\begin{equation}\label{eq:semistable-hypercohomology}
 \mathbb H^1(E,\mcK_{\mathrm{ss}}^\bullet)
 \xrightarrow{\sim}\mathbb H^1(E,\mathcal D^\bullet).
\end{equation}
Combining \eqref{eq:parabolic-hypercohomology} and
\eqref{eq:semistable-hypercohomology},
we obtain
\begin{equation}
 \mathbb H^1(E,\mcK_{\mathrm{par}}^\bullet)
 \simeq\mathbb H^1(E,\mcK_{\mathrm{ss}}^\bullet).
\end{equation}

Third, we compute the hypercohomology on the semistable side.  
Note that the endomorphism bundle decomposes as
\[
 \End(V)=\bigoplus_{i,j=1}^m M_{ij},
 \qquad M_{ij}:=L_{x_j}^{\vee}\otimes L_{x_i}.
\]
On the $(i,j)$ block, the Higgs differential is
\begin{equation}
 M_{ij}\xrightarrow{[\theta,\cdot]}M_{ij}\otimes K_E,
 \qquad u\longmapsto(\xi_i-\xi_j)u\,dz.
\end{equation}
If $i\ne j$ and $\xi_i\ne\xi_j$, this is an isomorphism of line bundles.  
If $i\ne j$ and $\xi_i=\xi_j$, distinctness forces $x_i\ne x_j$, 
so $M_{ij}$ is a nontrivial degree-zero line bundle.
In that case, we have $H^0(E,M_{ij})=0, H^1(E,M_{ij})=H^0(E,M_{ij}^{\vee})^*=0$.
Thus every off-diagonal complex has zero hypercohomology.
Hence, we have
\begin{equation}\label{eq:tangent-splitting}
 \mathbb H^1\bigl(E,\mcK_{\mathrm{par}}^\bullet\bigr)
 \simeq
 \mathbb H^1\bigl(E,\mcK_{\mathrm{ss}}^\bullet\bigr)
 \simeq\bigoplus_{i=1}^m
 \bigl(H^1(E,\mcO_E)\oplus H^0(E,K_E)\bigr).
\end{equation}

Finally, we can explicitly calculate the symplectic form.
Let $Z\in U_m$, its tangent space can be written as
\[
 T_ZU_m=\bigoplus_{a\in Z}T_a(T^*E).
\]
We write $a=(x(a),\xi(a)dz|_{x(a)}),\ u=(u_a)_{a\in Z},\ u_a=(\dot x(a),\dot\xi(a))$,
and for simplicity, let us denote $\delta_a:=d\mathrm{AJ}_{x(a)}(\dot x(a))\in H^1(E,\mcO_E)$.
Now, use \eqref{eq:jacobian-serre} and \eqref{eq:trivial-group-splitting},
together with \eqref{eq:tangent-splitting}, we obtain
\begin{align*}
 (\Phi_{1,m}^*\Theta_{\mcM_{1,m}})_Z(u)
 &=\sum_{a\in Z}\xi(a)\,
   \bigl\langle\delta_a,dz\bigr\rangle\\
 &=\sum_{a\in Z}\xi(a)\,
   dz|_{x(a)}(\dot x(a))
 =\sum_{a\in Z}\lambda_{E,a}(u_a).
\end{align*}
Therefore, for
$v=(v_a)_{a\in Z}$ with
$v_a=(\dot x'(a),\dot\xi'(a))$, we have
\begin{align}
 (\Phi_{1,m}^*\Omega_{\mcM_{1,m}})_Z(u,v)
 &=\sum_{a\in Z}(d\xi\wedge dz)_a(u_a,v_a) \notag\\
 &=\sum_{a\in Z}
 \left(
 \dot\xi(a)\,dz|_{x(a)}(\dot x'(a))
 -\dot\xi'(a)\,dz|_{x(a)}(\dot x(a))
 \right) \notag\\
 &=\sum_{a\in Z}\omega_{E,a}(u_a,v_a).
\end{align}
The last term is exactly the natural symplectic form on the reduced locus used by 
Beauville in his construction of the symplectic form on the Hilbert scheme \cite{Bea}.
Hence the two regular forms agree on the dense open set $U_m$, 
and therefore on the smooth irreducible variety $(T^*E)^{[m]}$.
\end{proof}

\subsection{The four non-trivial quotient cases}

Now we assume $|\Gamma|>1$, 
recall that we have denoted that $C=E/\Gamma\simeq\mbP^1$, and let $D\subset C$ be the branch divisor.  
If we denote the orbicurve as $\mathcal X_\Gamma=[E/\Gamma]$,
then the orbifold ($\Gamma$-equivariant)--parabolic correspondence \cite[\S2.5]{Groechenig}
identifies an orbifold Higgs bundle
\[
 (\mathcal F,\theta_{\mathcal F}),
 \qquad
 \theta_{\mathcal F}:\mathcal F\longrightarrow
 \mathcal F\otimes K_{\mathcal X_\Gamma},
\]
with a parabolic Higgs bundle
\[
 (F_*,\theta_F),
 \qquad
 \theta_F:F\longrightarrow F\otimes K_C(D),
\]
on $(C,D)$.

Indeed, let $\tau:\mathcal X_\Gamma\to C$ be the coarse moduli space morphism.  
For corresponding Higgs bundles $(\mathcal F,\theta_{\mathcal F})$ and
$(F_*,\theta_F)$, their underlying bundles satisfy the following canonical
isomorphisms in $\operatorname{Coh}(C)$:
\begin{equation}
 \begin{aligned}
 \tau_*\End(\mathcal F)
 &\simeq\ParEnd(F_*),\\
 \tau_*\bigl(\End(\mathcal F)\otimes K_{\mathcal X_\Gamma}\bigr)
 &\simeq\SParEnd(F_*)\otimes K_C(D).
 \end{aligned}
\end{equation}
We write
\[
 \mathcal K_{\mathrm{par}}^\bullet(F_*,\theta_F)
 =\bigl[\ParEnd(F_*)\xrightarrow{[\theta_F,\,\cdot\,]}
         \SParEnd(F_*)\otimes K_C(D)\bigr]
\]
and
\[
 \mathcal K_{\mathrm{orb}}^\bullet(\mathcal F,\theta_{\mathcal F})
 =\bigl[\End(\mathcal F)\xrightarrow{[\theta_{\mathcal F},\,\cdot\,]}
         \End(\mathcal F)\otimes K_{\mathcal X_\Gamma}\bigr].
\]

\begin{lemma}\label{lem:orbifold-parabolic-hypercohomology}
Let $\rho:E\to\mathcal X_\Gamma=[E/\Gamma]$ be the canonical morphism.
The orbifold--parabolic correspondence induces natural isomorphisms
\begin{equation}
 \begin{aligned}
 \mathbb H^1\bigl(C,\mathcal K_{\mathrm{par}}^\bullet(F_*,\theta_F)\bigr)
 &\simeq
 \mathbb H^1\bigl(\mathcal X_\Gamma,
                   \mathcal K_{\mathrm{orb}}^\bullet(\mathcal F,\theta_{\mathcal F})\bigr)\\
 &\simeq
 \mathbb H^1\bigl(E,
                  \rho^*\mathcal K_{\mathrm{orb}}^\bullet(\mathcal F,\theta_{\mathcal F})
            \bigr)^\Gamma.
 \end{aligned}
\end{equation}
These isomorphisms are compatible with the Serre-dual pairings in the following way: 
if $u,v$ correspond to $\widetilde u,\widetilde v$ under the above isomorphisms, then we have
\begin{equation}\label{eq:orbifold-comparison}
 \langle u,v\rangle_C
 =\langle\widetilde u,\widetilde v\rangle_{\mathcal X_\Gamma}
 =\frac1{|\Gamma|}
  \langle\rho^*\widetilde u,\rho^*\widetilde v\rangle_E.
\end{equation}
\end{lemma}

\begin{proof}
Note that we have the identification
\[
 \tau_*\mathcal K_{\mathrm{orb}}^\bullet(\mathcal F,\theta_{\mathcal F})
 \simeq \mathcal K_{\mathrm{par}}^\bullet(F_*,\theta_F).
\]
Since $\mathcal X_\Gamma$ is tame, $R^q\tau_*=0$ for $q>0$.
By comparing their Leray spectral sequences, we can conclude that
\[
 \mathbb H^1\bigl(\mathcal X_\Gamma,
                   \mathcal K_{\mathrm{orb}}^\bullet(\mathcal F,\theta_{\mathcal F})\bigr)
 \simeq
 \mathbb H^1\bigl(C,\mathcal K_{\mathrm{par}}^\bullet(F_*,\theta_F)\bigr).
\]

For the morphism $\rho:E\to\mathcal X_\Gamma=[E/\Gamma]$, we have the Cartan--Leray spectral sequence
\[
 H^p_{}\!\left(
   \Gamma,
   \mathbb H^q\bigl(E,
      \rho^*\mathcal K_{\mathrm{orb}}^\bullet(\mathcal F,\theta_{\mathcal F})\bigr)
 \right)
 \Longrightarrow
 \mathbb H^{p+q}\bigl(\mathcal X_\Gamma,
      \mathcal K_{\mathrm{orb}}^\bullet(\mathcal F,\theta_{\mathcal F})\bigr).
\]
Since $\Gamma$ is finite and we work over $\mbC$, for every
$\Gamma$-representation $W$ the operator
\[
 W\longrightarrow W^\Gamma,
 \qquad
 w\longmapsto\frac1{|\Gamma|}
 \sum_{\gamma\in\Gamma}\gamma w,
\]
is a projection onto $W^\Gamma$,
therefore, the functor $(-)^\Gamma$ is exact.
Hence, we have
\[
 H^p(\Gamma,W)=0\qquad(p>0).
\]
The $E_2$-page of the above spectral sequence is therefore
\[
 E_2^{p,q}=
 \begin{cases}
  \mathbb H^q\bigl(E,
   \rho^*\mathcal K_{\mathrm{orb}}^\bullet
   (\mathcal F,\theta_{\mathcal F})\bigr)^\Gamma,&p=0,\\
  0,&p>0.
 \end{cases}
\]
Which implies that the spectral sequence collapses at $E_2$.  

Thus, we have the isomorphism
\[
 \mathbb H^1\bigl(\mathcal X_\Gamma,
                   \mathcal K_{\mathrm{orb}}^\bullet(\mathcal F,\theta_{\mathcal F})\bigr)
 \simeq
 \mathbb H^1\bigl(E,
      \rho^*\mathcal K_{\mathrm{orb}}^\bullet(\mathcal F,\theta_{\mathcal F})\bigr)^\Gamma.
\]

Finally, under $\tau_*K_{\mathcal X_\Gamma}\simeq K_C$, we have
\begin{equation}
 \int_{\mathcal X_\Gamma}\eta
 =\frac1{|\Gamma|}\int_E\rho^*\eta,
 \qquad
 \eta\in H^1(\mathcal X_\Gamma,K_{\mathcal X_\Gamma}),
\end{equation}
which proves \eqref{eq:orbifold-comparison}.
\end{proof}

\begin{proposition}\label{prop:surface-symplectomorphism}
The isomorphism between the surfaces
\[
 \Phi_{\Gamma,1}:S_\Gamma\xrightarrow{\sim}\mcM_{\Gamma,1}
\]
is a symplectomorphism:
\[
 \Phi_{\Gamma,1}^*\Omega_{\mcM_{\Gamma,1}}=\omega_{S_\Gamma}.
\]
\end{proposition}

\begin{proof}
Set $S_\Gamma^\circ=\pi_\Gamma^{-1}(X_\Gamma^\circ)$ and fix $s\in S_\Gamma^\circ$.
Choose $a=(x,\xi dz|_x)\in(T^*E)^\circ$ with $q_\Gamma(a)=\pi_\Gamma(s)$.  
Note that we have the identification
\[
 T_sS_\Gamma
 \simeq
 \mathbb H^1\!\left(E,
 [\mathcal O_E\xrightarrow{0}K_E]\right)
 =
 H^1(E,\mathcal O_E)\oplus H^0(E,K_E).
\]
For $u,v\in T_sS_\Gamma$, we write the corresponding classes as 
$u_E=(\alpha,\beta),\ v_E=(\alpha',\beta')$,
where $\alpha,\alpha'\in H^1(E,\mathcal O_E),\ \beta,\beta'\in H^0(E,K_E)$.
Thus, from the identifications
$q_\Gamma^*\omega_{X_\Gamma^\circ}=\omega_E$ and
$\omega_{S_\Gamma}|_{S_\Gamma^\circ}
=(\pi_\Gamma|_{S_\Gamma^\circ})^*\omega_{X_\Gamma^\circ}$, we have
\begin{equation*}
 \omega_{S_\Gamma,s}(u,v)
 =
 \int_E\left(
  \alpha'\cup\beta-\alpha\cup\beta'
 \right)
 =:\langle u_E,v_E\rangle_E.
\end{equation*}

Next, for $(F_{s,*},\theta_s)=\Phi_{\Gamma,1}(s)$,
let $(\mathcal F_s,\theta_s)$ be the corresponding orbifold Higgs bundle.  
Its pullback to $E$ can be computed as follows
\[
 \rho^*(\mathcal F_s,\theta_s)
 \simeq
 \bigoplus_{\gamma\in\Gamma}
 (L_{\gamma x},\varphi_{\gamma a}),
 \qquad
 \varphi_{\gamma a}
 =\eta_\gamma\operatorname{id}_{L_{\gamma x}},
 \quad \eta_\gamma\in H^0(E,K_E).
\]
Let $M_{\gamma,\delta}:=\Hom(L_{\delta x},L_{\gamma x})$,
then we have
\begin{equation}\label{eq:pulled-back-decomposition}
 \rho^*\mathcal K_{\mathrm{orb}}^\bullet(\mathcal F_s)
 \simeq
 \bigoplus_{\gamma,\delta\in\Gamma}
 \left[
 M_{\gamma,\delta}
 \xrightarrow{\ (\eta_\gamma-\eta_\delta)\cdot\ }
 M_{\gamma,\delta}\otimes K_E
 \right].
\end{equation}
Since $K_E\simeq\mathcal O_E$ and the orbit $\{\gamma a:\gamma\in\Gamma\}$ is free,
for $\gamma\ne\delta$ we have
\[
 \begin{cases}
  \eta_\gamma\ne\eta_\delta
  &\Longrightarrow
  M_{\gamma,\delta}\xrightarrow{\sim}
  M_{\gamma,\delta}\otimes K_E,\\[2mm]
  \eta_\gamma=\eta_\delta
  &\Longrightarrow
  \gamma x\ne\delta x,\quad
  M_{\gamma,\delta}\not\simeq\mathcal O_E,\quad
  H^0(E,M_{\gamma,\delta})=H^1(E,M_{\gamma,\delta})=0.
 \end{cases}
\]
Hence every off-diagonal summand in \eqref{eq:pulled-back-decomposition} is acyclic.
Therefore, we have
\begin{equation}\label{eq:fiberwise-hypercohomology}
 \mathbb H^1\!\left(
 E,\rho^*\mathcal K_{\mathrm{orb}}^\bullet(\mathcal F_s)
 \right)
 \simeq
 \bigoplus_{\gamma\in\Gamma}
 \mathbb H^1\!\left(E,
 [\mathcal O_E\xrightarrow{0}K_E]\right).
\end{equation}
As a consequence, for $c_\gamma=(\alpha_\gamma,\beta_\gamma),\ c'_\gamma=(\alpha'_\gamma,\beta'_\gamma)$,
the pairing on \eqref{eq:fiberwise-hypercohomology} can be written as
\begin{equation}\label{eq:fiberwise-pairing}
 \begin{split}
 \bigl\langle(c_\gamma)_\gamma,(c'_\gamma)_\gamma\bigr\rangle_E
 &=
 \sum_{\gamma\in\Gamma}
 \int_E\left(
  \alpha'_\gamma\cup\beta_\gamma
  -\alpha_\gamma\cup\beta'_\gamma
 \right)\\
 &=\sum_{\gamma\in\Gamma}
   \langle c_\gamma,c'_\gamma\rangle_E.
 \end{split}
\end{equation}

On the other hand, from Lemma~\ref{lem:orbifold-parabolic-hypercohomology}, we have
\begin{equation*}
 \begin{aligned}
 \mathbb H^1\!\left(
 C,\mathcal K_{\mathrm{par}}^\bullet(F_{s,*})
 \right)
 &\simeq
 \left(
 \bigoplus_{\gamma\in\Gamma}
 \mathbb H^1\!\left(E,
 [\mathcal O_E\xrightarrow{0}K_E]\right)
 \right)^\Gamma\\
 &=\left\{
  (\gamma\cdot c)_{\gamma\in\Gamma}:
  c\in
  \mathbb H^1\!\left(E,
  [\mathcal O_E\xrightarrow{0}K_E]\right)
 \right\}.
 \end{aligned}
\end{equation*}

Let $U,V\in \mathbb H^1(C,\mathcal K_{\mathrm{par}}^\bullet(F_{s,*}))$
be the parabolic first-order deformations corresponding under $\Phi_{\Gamma,1}$ to $u,v$, 
and let $\widetilde U,\widetilde V$ be their orbifold ones.
Since we have
\begin{equation*}
 \rho^*\widetilde U
 =(\gamma\cdot u_E)_{\gamma\in\Gamma},
 \qquad
 \rho^*\widetilde V
 =(\gamma\cdot v_E)_{\gamma\in\Gamma},
\end{equation*}
and the action of $\Gamma$ preserves the pairing:
\[
 \langle\gamma\cdot u_E,\gamma\cdot v_E\rangle_E
 =\langle u_E,v_E\rangle_E.
\]
Therefore, by \eqref{eq:orbifold-comparison} and
\eqref{eq:fiberwise-pairing}, we have
\begin{align*}
 (\Phi_{\Gamma,1}^*\Omega_{\mcM_{\Gamma,1}})_s(u,v)
 &=\langle U,V\rangle_C\\
 &=\frac1{|\Gamma|}
   \langle\rho^*\widetilde U,\rho^*\widetilde V\rangle_E\\
 &=\frac1{|\Gamma|}\sum_{\gamma\in\Gamma}
   \langle\gamma\cdot u_E,\gamma\cdot v_E\rangle_E\\
 &=\langle u_E,v_E\rangle_E\\
 &=\omega_{S_\Gamma,s}(u,v).
\end{align*}
Thus the two regular forms agree on $S_\Gamma^\circ$, and hence on
$S_\Gamma$.
\end{proof}

Finally, we can consider cases for higher dimensional, i.e. $m\geq2$.
Let $U_m\subset S_\Gamma^{[m]}$ be the locus of reduced subschemes
\[
 Z=\{s_1,\ldots,s_m\},
 \qquad s_i\in S_\Gamma^\circ,\quad s_i\ne s_j\ (i\ne j).
\]
Note that we have the decomposition of the tangent space
\[
 T_ZS_\Gamma^{[m]}\simeq\bigoplus_{i=1}^mT_{s_i}S_\Gamma.
\]
Thus, for $u=(u_i)$ and $v=(v_i)$, 
the symplectic form (see \cite{Bea}) induced from 
the natural symplectic form of cotangent bundle satisfies
\begin{equation}
 \omega_{S_\Gamma^{[m]},Z}(u,v)
 =\sum_{i=1}^m\omega_{S_\Gamma,s_i}(u_i,v_i).
\end{equation}

\begin{proposition}
On $U_m$ one has
\begin{equation}
 \left.\Phi_{\Gamma,m}^*\Omega_{\mcM_{\Gamma,m}}\right|_{U_m}
 =\left.\omega_{S_\Gamma^{[m]}}\right|_{U_m}.
\end{equation}
\end{proposition}

\begin{proof}
 This is due to some similar calculations as in the proof of 
 Proposition \ref{prop:surface-symplectomorphism} 
 and Proposition \ref{prop:trivial-group-symplectomorphism}.
\end{proof}

Now we reach to the main theorem.

\begin{theorem}
Let $\Gamma\subset\Aut(E,o)$ be a finite cyclic group with
$|\Gamma|\in\{1,2,3,4,6\}$, and let
\[
 \Phi_{\Gamma,m}:S_\Gamma^{[m]}
 \xrightarrow{\sim}\mcM_{\Gamma,m}
\]
be the isomorphism from the Hilbert scheme of $m$ points on $S_\Gamma=\Gamma\text{-}\Hilb(T^*E)$ 
to the corresponding stable parabolic Higgs moduli space on $C_\Gamma$.  
Then we have
\[
 \Phi_{\Gamma,m}^*\Omega_{\mcM_{\Gamma,m}}
 =\omega_{S_\Gamma^{[m]}},
\]
where $\Omega_{\mcM_{\Gamma,m}}$ is the symplectic form on the moduli
space of stable parabolic Higgs bundles constructed by deformation
theory, and $\omega_{S_\Gamma^{[m]}}$ is the holomorphic symplectic
form induced by the natural symplectic form on the surface $S_\Gamma$.
\end{theorem}

\begin{proof}
Note that $U_m$ is open dense in the smooth irreducible variety $S_\Gamma^{[m]}$,
and both $\Phi_{\Gamma,m}^*\Omega_{\mcM_{\Gamma,m}}$ and $\omega_{S_\Gamma^{[m]}}$ are regular, 
their equality on $U_m$ therefore extends uniquely to all of $S_\Gamma^{[m]}$.  
This proves the case for $|\Gamma|>1$ 
and Proposition \ref{prop:trivial-group-symplectomorphism} proves the case for $|\Gamma|=1$.
\end{proof}

\subsection{Compatibility with the Hitchin maps}\label{sec:hitchin-maps}

In this final subsection, we compare the Hitchin maps on both sides of the isomorphism $\Phi_{\Gamma,m}$.
We fix the trivialization and write
\[
 T^*E\simeq E\times\mbA^1_\xi,
 \qquad
 (x,\xi)\longleftrightarrow \xi\,dz|_x.
\]
Choose a generator $\gamma_0$ of $\Gamma$ and set $\zeta=\zeta_{\gamma_0}$. 
In the case of $|\Gamma|=1$, one can just take $\gamma_0=1$ and $\zeta=1$.
Let
\[
 p_\xi:T^*E\longrightarrow\mbA^1_\xi,
 \qquad
 p_\xi(x,\xi)=\xi,
\]
be the projection.  
Note that $p_\xi$ is $\Gamma$-equivariant with respect to the action of
$\gamma\in\Gamma$ on $\mbA^1_\xi$ by $\xi\mapsto\zeta_\gamma^{-1}\xi$.
Let
\[
 \nu_{|\Gamma|}:\mbA^1_\xi\longrightarrow\mbA^1_t,
 \qquad
 \nu_{|\Gamma|}(\xi)=t=\xi^{|\Gamma|}.
\]
The composition
\[
 f_\Gamma:=\nu_{|\Gamma|}\circ p_\xi:T^*E\longrightarrow\mbA^1_t
\]
is a regular $\Gamma$-invariant morphism.
By the universal property of the finite quotient
$q_\Gamma:T^*E\to X_\Gamma$,
there is a unique regular morphism
\[
 \overline h_\Gamma:X_\Gamma\longrightarrow\mbA^1_t
\]
satisfying
\[
 \overline h_\Gamma\circ q_\Gamma=f_\Gamma,
 \qquad
 \overline h_\Gamma\bigl(q_\Gamma(x,\xi)\bigr)=\xi^{|\Gamma|}.
\]
Since $\pi_\Gamma:S_\Gamma\to X_\Gamma$ is a morphism, the composition
\[
 h_\Gamma:=\overline h_\Gamma\circ\pi_\Gamma:
 S_\Gamma\longrightarrow\mbA^1_t
\]
is a regular morphism.

Thus we can construct the composition
\begin{equation}\label{eq:hilbert-hitchin-map}
 h_{\Gamma,m}:
 S_\Gamma^{[m]}
 \xrightarrow{\ \mathrm{H}\ }S_\Gamma^{(m)}
 \xrightarrow{\ h_\Gamma^{(m)}\ }(\mbA^1_t)^{(m)}
 \longrightarrow \mbA^m,
\end{equation}
where $\mathrm{H}$ is the Hilbert--Chow morphism.

\begin{definition}[Characteristic polynomial]
Put $N=|\Gamma|\,m$.  
For a parabolic Higgs bundle $(F_*,\theta)\in\mcM_{\Gamma,m}$, 
we define the characteristic polynomial of $\theta$ by 
\[
 \det(T\operatorname{id}_F-\theta)
 =
 T^N+b_1(\theta)T^{N-1}+\cdots+b_N(\theta),
\]
where
\begin{equation}\label{eq:coefficients}
 b_i(\theta)=(-1)^i\tr\!\left(\bigwedge^i\theta\right), \quad 1\leq i\leq N.
\end{equation}
\end{definition}

Recall that, for $|\Gamma|>1$, the integer $a_p$ denotes the stabilizer order at $p\in D$.  
Let us set $N=|\Gamma|m,\ 0=p_\Gamma(o)\in C_\Gamma$, and denote
\[
 \delta_{\Gamma,m}(j)
 =
 \begin{cases}
  1,
  &
  j=k|\Gamma|+1\text{ for }1\leq k\leq m-1,
  \\[1mm]
  0,
  &
  \text{otherwise}.
 \end{cases}
\]

If we define a line bundle by
\begin{equation}\label{eq:refined-hitchin-line-bundles}
 \mathcal L_{\Gamma,j}
 :=
 \begin{cases}
  K_E,
  &
  |\Gamma|=1,\ j=1,
  \\[1mm]
  K_E^{\otimes j}(o),
  &
  |\Gamma|=1,\ 2\leq j\leq m,
  \\[1mm]
  K_{C_\Gamma}^{\otimes j}
  \left(
   \displaystyle
   \sum_{p\in D}
   \left\lfloor
    j\left(1-\frac1{a_p}\right)
   \right\rfloor p
   +
   \delta_{\Gamma,m}(j)\,0
  \right),
  &
  |\Gamma|>1.
 \end{cases}
\end{equation}

Then, we have 
$
 b_j(\theta)\in
 H^0(C_\Gamma,\mathcal L_{\Gamma,j}),
 1\leq j\leq N.
$
Therefore, we can write the Hitchin base as following
\begin{equation}\label{eq:parabolic-hitchin-base}
 \mathcal B_{\Gamma,m}^{\mathrm{par}}
 =
 \bigoplus_{j=1}^{N}
 H^0(C_\Gamma,\mathcal L_{\Gamma,j})\simeq\mbA^m.
\end{equation}

Indeed, fix $0\ne dz\in H^0(E,K_E)$. 
For $|\Gamma|>1$ and the quotient map $p_\Gamma:E\to C_\Gamma$,
by the Riemann--Hurwitz formula, we have
$p_\Gamma^*\mathcal L_{\Gamma,|\Gamma|}\simeq K_E^{\otimes |\Gamma|}$.
Since $(dz)^{|\Gamma|}$ is $\Gamma$-invariant, 
we can find the following section
\[
 \kappa_\Gamma\in H^0(C_\Gamma,\mathcal L_{\Gamma,|\Gamma|}),
 \qquad p_\Gamma^*\kappa_\Gamma=(dz)^{|\Gamma|}.
\]
For $|\Gamma|>1$, we can calculate the degree of the line bundle as
\[
 \deg\mathcal L_{\Gamma,j}
 =
 \begin{cases}
  0,&|\Gamma|\mid j,\\
  -2,&j=1,\\
  -1,&|\Gamma|\nmid j,\ j>1.
 \end{cases}
\]
Therefore, we have
\[
 H^0(C_\Gamma,\mathcal L_{\Gamma,j})
 =
 \begin{cases}
  \mathbb C\,\kappa_\Gamma^k,
  &
  j=k|\Gamma|,\\
  0,
  &
  |\Gamma|\nmid j.
 \end{cases}
\]
Thus the isomorphism in \eqref{eq:parabolic-hitchin-base} can be written explicitly as
\begin{equation}\label{eq:parabolic-hitchin-coordinates}
 \begin{aligned}
 \mathcal B_{\Gamma,m}^{\mathrm{par}}&\xrightarrow{\ \sim\ }\mbA^m,\\
 (b_1,\ldots,b_N)&\longmapsto(c_1,\ldots,c_m),
 \qquad b_{k|\Gamma|}=c_k\kappa_\Gamma^k.
 \end{aligned}
\end{equation}
Here $b_j=0$ whenever $|\Gamma|\nmid j$.

\begin{definition}[Hitchin map]
We define the Hitchin map as
\[
 \begin{aligned}
 \operatorname{Hit}_{\Gamma,m}:\mcM_{\Gamma,m}
 &\longrightarrow\mathcal B_{\Gamma,m}^{\mathrm{par}},\\
 (F_*,\theta)&\longmapsto
 \bigl(b_1(\theta),\ldots,b_N(\theta)\bigr).
 \end{aligned}
\]
\end{definition}

\begin{theorem}[Compatibility with the Hitchin maps]
For every finite cyclic group $\Gamma\subset\Aut(E,o)$ with
$|\Gamma|\in\{1,2,3,4,6\}$, 
the isomorphism $\Phi_{\Gamma,m}$ is compatible with the Hitchin map.  
More precisely, the diagram
\[
 \begin{tikzcd}[column sep=large,row sep=large]
 S_\Gamma^{[m]}
   \arrow[r,"\Phi_{\Gamma,m}"]
   \arrow[d,"h_{\Gamma,m}"']&
 \mcM_{\Gamma,m}
   \arrow[d,"\operatorname{Hit}_{\Gamma,m}"]\\
 \mbA^m\arrow[r,equal]&\mbA^m
 \end{tikzcd}
\]
is commutative.
\end{theorem}

\begin{proof}
It suffices to compare the two morphisms on the dense open subset $U_m$.
Let
\[
 Z=\{s_1,\ldots,s_m\}\in U_m,
 \qquad
 a_i=(x_i,\xi_i\,dz|_{x_i})\in(T^*E)^\circ,
 \qquad q_\Gamma(a_i)=\pi_\Gamma(s_i),
\]
and put $t_i=h_\Gamma(s_i)=\xi_i^{|\Gamma|}$.  Write
\begin{equation}\label{eq:product-coordinates}
 \prod_{i=1}^m(Y-t_i)
 =Y^m+c_1Y^{m-1}+\cdots+c_m.
\end{equation}
By \eqref{eq:hilbert-hitchin-map},
$h_{\Gamma,m}(Z)=(c_1,\ldots,c_m)$.

Let $(F_{i,*},\theta_i)=\Phi_{\Gamma,1}(s_i)$,
we have the following 
\[
 \begin{aligned}
 p_\Gamma^*\det(T\operatorname{id}_{F_i}-\theta_i)
 &=\prod_{\ell=0}^{|\Gamma|-1}
   (T-\zeta^{-\ell}\xi_i\,dz)\\
 &=T^{|\Gamma|}-t_i(dz)^{|\Gamma|}
 =p_\Gamma^*(T^{|\Gamma|}-t_i\kappa_\Gamma).
 \end{aligned}
\]
Therefore, we have
\[
 \det(T\operatorname{id}_{F_i}-\theta_i)
 =T^{|\Gamma|}-t_i\kappa_\Gamma.
\]
On $U_m$, additivity of the relative Fourier--Mukai transform gives, 
after forgetting the parabolic flags,
\[
 (F_Z,\theta_Z)=\bigoplus_{i=1}^m(F_i,\theta_i),
 \qquad (F_{Z,*},\theta_Z)=\Phi_{\Gamma,m}(Z).
\]
Consequently, using \eqref{eq:product-coordinates}, we have
\[
 \begin{aligned}
 \det(T\operatorname{id}_{F_Z}-\theta_Z)
 &=\prod_{i=1}^m(T^{|\Gamma|}-t_i\kappa_\Gamma)\\
 &=T^{|\Gamma|m}+c_1\kappa_\Gamma T^{|\Gamma|(m-1)}
   +\cdots+c_m\kappa_\Gamma^m.
 \end{aligned}
\]
By \eqref{eq:parabolic-hitchin-coordinates}, we have the following identification
\[
 \operatorname{Hit}_{\Gamma,m}(\Phi_{\Gamma,m}(Z))
 =(c_1,\ldots,c_m)=h_{\Gamma,m}(Z).
\]
Both sides are regular morphisms, 
and $U_m$ is dense in the irreducible variety $S_\Gamma^{[m]}$, 
hence they agree everywhere.
\end{proof}



\end{document}